\documentclass[10pt]{amsart}
\usepackage{amsmath}
  \usepackage{graphics} 
  \usepackage{epsfig} 
 \usepackage[colorlinks=true]{hyperref}
  \usepackage{float}
  \usepackage{caption}
\newtheorem{theorem}{Theorem}[section]
\newtheorem{corollary}{Corollary}

\newtheorem{proposition}{Proposition}

\theoremstyle{definition}

\newtheorem{remark}{Remark}

\title[ three species food chain model]
      {A remark on  ``Study of a Leslie-Gower-type tritrophic population model" [Chaos, Solitons and Fractals 14 (2002) 1275-1293]}

\author[Parshad, Kumari, Kouachi ]{}


 \email{rparshad@clarkson.edu}
 \email{nkumari@clarkson.edu}
 \email{kouachi@hotmail.com}

\begin{document}
\maketitle

\centerline{\scshape Rana D. Parshad$^{1}$, Nitu Kumari$^{1,2}$ and Said Kouachi$^{3}$
 }
\medskip
{\footnotesize
 \centerline{1) Department of Mathematics,}
 \centerline{Clarkson University,}
   \centerline{ Potsdam, New York 13699, USA.}
   \medskip
   \centerline{ 2) School of Basic Sciences,}
 \centerline{Indian Institute of Technology Mandi}
   \centerline{ Mandi, Himachal Pradesh 175001, India.}
   \medskip
 \centerline{ 3) Department of Mathematics, College of Science,}
 \centerline{Qassim University,}
   \centerline{ Al-Gassim, Buraydah 51452, Kingdom of Saudi Arabia}
   
 }

\begin{abstract}
In \cite{AA02} a three species ODE model, based on a modified Leslie-Gower scheme is investigated. It is shown that under certain restrictions on the parameter space, the model has bounded solutions for all positive initial conditions, which eventually enter an invariant attracting set. We show that this is not true. To the contrary, solutions to the model can blow up in finite time, even under the restrictions derived in \cite{AA02}, if the initial data is large enough. We also prove similar results for the spatially extended system. We validate all of our results via numerical simulations.
 \end{abstract}

 \section{Introduction}
 \label{1}

The purpose of this letter is to remark on the well cited research article \cite{AA02},
where the following tri-trophic population model originally proposed in \cite{U98, U97} is considered,

\begin{equation}
\label{eq:x1o}
\frac{du}{dt}=  a_{1}u-b_{1}u^{2}-w_{0}\left(\frac{u v }{u+D_{0}}\right),
\end{equation}
\begin{equation}
\label{eq:x2o}
\frac{dv}{dt}=  -a_{2}v+w_{1}\left(\frac{u v}{u + D_{1}}\right)-w_{2}\left(\frac{vr}{v+D_{2}}\right),
\end{equation}
\begin{equation}
\label{eq:x3o}
\frac{dr}{dt}=  cr^{2}-w_{3}\frac{r^{2}}{v+D_{3}}.
\end{equation}

This model is based on the Leslie-Gower formulation \cite{L48}, and considers the interactions between a generalist top predator $r$, specialist middle predator $v$, and prey $u$, where $(u,v,r)$ are solutions to the above system \eqref{eq:x1o}-\eqref{eq:x3o}. 
The model is very rich dynamically, and has led to a number of works in the literature \cite{L02, N13, P10, PK13, UIR00, P13, U7,U5, G05}. In \cite{AA02} various theorems on the boundedness of the system \eqref{eq:x1o}-\eqref{eq:x3o} are proved, and the existence of an invariant attracting set is established.
In particular, we recall the following result (Theorem 3 (2i), (3i) \cite{AA02})

\begin{theorem}
\label{thm:aziz}
Consider the model \eqref{eq:x1o}-\eqref{eq:x3o}. Under the assumption that 

\begin{equation}
\label{eq:ac1}
c < \left(  \frac{w_0 b_1 D_3}{w_1\left( a_1 + \frac{(a_1)^2}{4a_2}\right) + w_0 b_1 D_3}\right) \frac{w_3}{D_3},
\end{equation}

all solutions to \eqref{eq:x1o}-\eqref{eq:x3o} are uniformly bounded forward in time, for any initial data in $\mathbb{R}^{3}_{+}$, and they eventually enter a bounded attracting set $\mathcal{A}$. Furthermore system \eqref{eq:x1o}-\eqref{eq:x3o} is dissipative.
\end{theorem}

Note, $\mathcal{A}$ is explicitly defined in \cite{AA02}. Also condition \eqref{eq:ac1} is equation (7) in \cite{AA02}. The form of \eqref{eq:ac1} is different from equation (7) in \cite{AA02}, as in \cite{AA02} different constants have been used, than what we are currently using. However it is a matter of simple algebra to convert equation (7) from \cite{AA02} to our setting. 

Our aim in the current letter is to show that (Theorem 3 (2i), (3i) \cite{AA02}) is incorrect. In particular we show,

\vspace{2mm}
1) Solutions to \eqref{eq:x1o}-\eqref{eq:x3o} are not bounded uniformly in time, even if condition \eqref{eq:ac1} from Theorem \ref{thm:aziz} is met. Furthermore, solutions to \eqref{eq:x1o}-\eqref{eq:x3o} can even blow-up in finite time for large initial data. Thus  there is no absorbing set $\mathcal{A}$ for all initial conditions in $\mathbb{R}^{3}_{+}$, and  system \eqref{eq:x1o}-\eqref{eq:x3o} is not dissipative in $\mathbb{R}^3_{+}$, even under condition \eqref{eq:ac1}.

\vspace{2mm}
2) Similar results hold for the spatially explicit model.

\vspace{2mm}
3) The above results can be validated numerically. We choose parameters satisfying \eqref{eq:ac1}, and show that numerical simulations of \eqref{eq:x1o}-\eqref{eq:x3o}, and its spatially extended form, still lead to finite time blow-up.

\section{Finite time blow-up in the ODE model}

We state the following theorem

\begin{theorem}
\label{thm:r1}
Consider the three species food chain model \eqref{eq:x1o}-\eqref{eq:x3o}. For 

\begin{equation}
\label{eq:l1r1}
c < \left(  \frac{w_0 b_1 D_3}{w_1\left( a_1 + \frac{(a_1)^2}{4a_2}\right) + w_0 b_1 D_3}\right) \frac{w_3}{D_3},
\end{equation}

 $r(t)$ blows up in finite time, that is 

\begin{equation}
\label{eq:bu}
\mathop{\lim}_{t \rightarrow T^{*} < \infty} |r(t)| \rightarrow \infty,  
 \end{equation}

as long as the initial data $v_0, r_0$ are large enough.
\end{theorem}

\begin{proof}

First set $\left(  \frac{w_0 b_1 D_3}{w_1\left( a_1 + \frac{(a_1)^2}{4a_2}\right) + w_0 b_1 D_3}\right) = k < 1$.

Consider the following modification to system \eqref{eq:x1o}-\eqref{eq:x3o}, with solution $(u_1,v_1,r_1)$,

\begin{equation}
\label{eq:x1om}
\frac{du_{1}}{dt}=  a_{1}u_1-b_{1}(u_{1})^{2}-w_{0}\left(\frac{u_{1} v_{1} }{u_{1}+D_{0}}\right),
\end{equation}
\begin{equation}
\label{eq:x2om}
\frac{dv_{1}}{dt}=  -a_{2}v_{1}-w_{2}v_{1}r_{1},
\end{equation}
\begin{equation}
\label{eq:x3om}
\frac{dr_1}{dt}=   \delta r_{1}^{2}.
\end{equation}

 Recall that $r_{1}$, the solution to \eqref{eq:x3om}, blows up in finite time, at $T^{**}=\frac{1}{\delta |r_1(0)|}$, and we have an exact solution for $r_{1}$, for $t \in \left[0, \frac{1}{\delta |r_1(0)|}\right)$,  given by,

\begin{equation}
\label{eq:x2e}
r_{1}= \frac{1}{\frac{1}{r_1(0)} - \delta t}.
\end{equation}

 However, using this exact solution for $r_{1}$, one can find an exact solution to \eqref{eq:x2om} via separation of variables. Thus 

\begin{equation}
\label{eq:x2e}
v_{1}= v_1(0) e^{-a_{2} t} \left(1-r_1(0)\delta t\right)^{\frac{w_2}{\delta}}, 
\end{equation}

for $t \in \left[0, \frac{1}{\delta |r_1(0)|}\right)$.
Next for a given $c,k,w_{3},D_{3}$ we choose $\delta$ s.t we can enforce 

\begin{equation}
\label{eq:x2es}
\frac{w_3 }{v_1+D_{3}} + \frac{\delta}{2} \leq c < \delta < k \frac{w_3 }{D_{3}}
\end{equation} 

to hold $\forall t$ s.t. $ t \in \left[0, \frac{1}{2\delta |r_1(0)|}\right]$.  \eqref{eq:x2es} implies 

\begin{equation}
\label{eq:c12}
\left(\frac{1}{\frac{k}{D_3}-\frac{\delta}{2w_3}} - D_3 \right) < v_1.
\end{equation}

(Here we assume $\left(\frac{1}{\frac{k}{D_3}-\frac{\delta}{2w_3}} - D_3 \right) > 0$, else it is an uninteresting case)

Equivalently we have
 
 \begin{equation}
\label{eq:resn1}
\left(\frac{1}{\frac{k}{D_3}-\frac{\delta}{2w_3}} - D_3 \right) <  v_1(0) e^{-a_{2} t} \left(1-r_1(0)\delta t\right)^{\frac{w_2}{\delta}}
 \end{equation}
 This of course is always possible for $v_1(0), r_{1}(0)$ large enough.
 
However, note that $v_{1}$ is a subsolution to \eqref{eq:x2o}. If $D_{2} \geq 1$, this is immediate as then $w_{2}\left(\frac{vr}{v+D_{2}}\right) < w_{2}vr$. If $D_2 < 1$, then we can assume $w_{2}\left(\frac{vr}{v+D_{2}}\right) < w_{4}vr$, where we select $w_4$ s.t $w_{4}(v+D_{2}) > w_{2}$, and then choose $w_4$ in place of $w_2$ in \eqref{eq:x2om}. Also $r_{1}$ (with $\frac{\delta}{2}$ in place of $\delta$ in \eqref{eq:x3om}) is a subsolution to \eqref{eq:x3o}, as long as \eqref{eq:x2es} holds. Thus via direct comparison, $v > v_{1}$ and $r > r_{1}$. Since \eqref{eq:x2es} implies $c - \frac{w_3 }{v_1+D_{3}} \geq \frac{\delta}{2}$, it is immediate that the solution $r$ to \eqref{eq:x3o} will also blow-up, via direct comparison with $r_1$ solving \eqref{eq:x3om} with $\frac{\delta}{2}$ in place of $\delta$.

See figure \ref{fig:ODE}, for a simple graphical representation of this idea. Thus we have ascertained the blow-up of system \eqref{eq:x1o}-\eqref{eq:x3o}, via direct comparison to the modified system \eqref{eq:x1om}-\eqref{eq:x3om}.
 This proves the Theorem.

\end{proof}

 We next state the following Theorem
 
\begin{theorem}
\label{thm:d1}
The three species food chain model \eqref{eq:x1o}-\eqref{eq:x3o}, even under condition \eqref{eq:l1r1} is not dissipative in all of $\mathbb{R}^{3}_{+}$.

\end{theorem}
\begin{proof}
Via Theorem \ref{thm:r1}, there exists initial data in $\mathbb{R}^{3}_{+}$, for which solutions blow-up in finite time, and thus do not enter any bounded attracting set $\mathcal{A}$. Thus system \eqref{eq:x1o}-\eqref{eq:x3o} is not dissipative.
\end{proof}

\begin{remark}
The essential error made in the proof in \cite{AA02} is in equation $(12)$ in \cite{AA02}. The derived bound for $v$ is inserted in an estimate for the sum of $u,v$ and $r$. Although it is true that $v$ is bounded, and enters an attracting set eventually, \emph{there is some transition time} before this happens. If $v_0$ is chosen arbitrarily large, then this transition time can be made arbitrarily long. The key is for $r_0,v_0$ chosen large enough, during this transition time, we can enforce $c-\frac{w_{3}}{v+D_{3}} > \delta_{1} > 0$, for as long as it takes $\frac{dr}{dt} = \delta_{1} r^2$ to bow up.
In this case $r$ will also blow-up in finite time, (by comparison to 
$\frac{dr}{dt} = \delta_{1} r^2$),  and hence never enter any bounded attracting set.
\end{remark}

\section{Finite time blow-up in the PDE model}

\subsection{Preliminaries}
We now consider the following spatially extended version of \eqref{eq:x1o}-\eqref{eq:x3o}
\begin{equation}
\label{eq:(1.1)}
\partial _{t}u-d_{1}\Delta u=f(u,v,r)=a_{1}u-b_{1}u^{2}-w_{0}\frac{uv}{u+D_{0}}, 
   \end{equation}%

\begin{equation}
\label{eq:(1.2)}
\partial _{t}v-d_{2}\Delta v=g(u,v,r)=-a_{2}v+w_{1}\frac{uv}{u+D_{1}}-w_{2}\frac{vr}{v+D_{2}},
\end{equation}%

\begin{equation}
\label{eq:(1.3)}
\partial _{t}r-d_{3}\Delta r=h(u,v,r)=cr^{2}-\omega _{3}\frac{r^{2}}{v+D_{3}},
\end{equation}%

  defined on $\mathbb{R}^{+}\times \Omega$. Here $\Omega \subset \mathbb{R}^{N}$ and 
where $a_{1},a_{2},b_{1},c,\ D_{0},\ D_{1},D_{2},\ D_{3},\ w_{0},\
w_{1}, w_{2}$ and $w_{3}$, the parameters in the problem as earlier, are positive constants.
We can prescribe either Dirichlet or Neumann boundary conditions.

%

$\Omega $ is an open bounded domain in $\mathbb{R}^{N}$ with smooth boundary $
\partial \Omega $. 
 $d_{1}$, $d_{2}$ and $d_{3}$ are the positive diffusion coefficients.
 
The initial data 
\begin{equation}
u(0,x)=u_{0}(x),\;v(0,x)=v_{0}(x)\text{,\ }r(0,x)=r_{0}(x)\ \;\;\ \text{in }%
\Omega   \tag*{(1.5)}
\end{equation}%
are assumed to be nonnegative and uniformly bounded on $\Omega $.
%
%

The nonnegativity of the solutions is preserved by application of classical
results on invariant regions (\cite{S83}), since the reaction terms are
 quasi-positive, i.e.%
\begin{equation}
\label{eq:2.1}
f\left( 0,v,r\right) \geq 0,\ g\left( u,0,r\right) \geq 0,\ h\left(
u,v,0\right) \geq 0\ ,\ \ \text{\ for all }u,v,r \geq 0. 
\end{equation}%

The usual norms in the spaces $\mathbb{L}^{p}(\Omega )$, $\mathbb{L}^{\infty
}(\Omega )$ and $\mathbb{C}\left( \overline{\Omega }\right) $ are
respectively denoted by

\[
\left\Vert u\right\Vert _{p}^{p}\text{=}\frac{1}{\left\vert \Omega
\right\vert }\int\limits_{\Omega }\left\vert u(x)\right\vert ^{p}dx,
\]%
\[
\left\Vert u\right\Vert _{\infty }\text{=}\underset{x\in \Omega }{max}%
\left\vert u(x)\right\vert .
\]

Since the reaction terms are continuously differentiable on $\mathbb{R}%
^{+3}$, then for any initial data in $\mathbb{C}\left( \overline{\Omega }%
\right) $ or $\mathbb{L}^{p}(\Omega ),\;p\in \left( 1,+\infty \right) $, it
is easy to check directly their Lipschitz continuity on bounded subsets of
the domain of a fractional power of the operator $I_{3}\left(
d_{1},d_{2},d_{3}\right) ^{t}\Delta $, where $I_{3}$ the three dimensional
identity matrix, $\Delta $ is the Laplacian operator and $\left( {}\right)
^{t}$ denotes the transposition. Under these assumptions, the following
local existence result is well known (see \cite{F64,H84,S83,P83,R84}).

\begin{proposition}
\label{prop:ls}
The system \eqref{eq:(1.1)}-\eqref{eq:(1.3)} admits a unique, classical solution $(u,v,r)$ on $%
[0,T_{\max }[\times \Omega $. If $T_{\max }<\infty $ then 
\begin{equation}
\underset{t\nearrow T_{\max }}{\lim }\left\{ \left\Vert u(t,.)\right\Vert
_{\infty }+\left\Vert v(t,.)\right\Vert _{\infty }+\left\Vert
r(t,.)\right\Vert _{\infty }\right\} =\infty . 
\end{equation}
\end{proposition}

\subsection{Blow-up}

Here we will show that \eqref{eq:(1.1)}-\eqref{eq:(1.3)}, blows up in finite time. We will do this by looking back at the blow-up for $r$ in \eqref{eq:x1o}, and then using a standard comparison method. Consider \eqref{eq:x1o}-\eqref{eq:x3o}, with initial conditions $u_{0}^{-},\ v_{0}^{-}$ and $r_{0}^{-}$ strictly positive.

By integrating the third equation of the ODE system, we have%
\[
-\frac{1}{r}+\frac{1}{r_{0}^{-}}=ct- w_{3}\underset{0}{\overset{t}{\int 
}}\frac{ds}{v+D_{3}},
\]%
which gives%
\[
r=\frac{1}{\frac{1}{r_{0}^{-}}-ct+ w_{3}\underset{0}{\overset{t}{\int }}%
\frac{ds}{v+D_{3}}}.
\]%
If we prove that the function: $t\rightarrow \psi \left( t\right) =\frac{1}{%
r_{0}^{-}}-ct+w_{3}\underset{0}{\overset{t}{\int }}\frac{dt}{v+D_{3}}$
vanishes at a time $T>0$ and since $\psi \left( 0\right) >0$, then the
solution will blow-up in finite time.\newline

\bigskip Since the reaction terms are continuous functions, then the
solutions are classical and continuous and%
\[
\underset{t\rightarrow 0}{\lim }\left( \frac{1}{t}\underset{0}{\overset{t}{%
\int }}\frac{ds}{v+D_{3}}\right) =\frac{1}{v_{0}^{-}+D_{3}}.
\]%
If $v_{0}^{-}$ is sufficiently large, then there exists $\delta >0$ such that%
\[
\frac{1}{t}\underset{0}{\overset{t}{\int }}\frac{ds}{v+D_{3}}<\frac{c}{
2w_{3}},\ \ \ \ \ \text{\ for all }t\in ( 0,\delta ) .
\]%
Then%
\[
\frac{1}{r_{0}^{-}}-ct + w_{3}\underset{0}{\overset{t}{\int }}\frac{ds}{%
v+D_{3}}=\frac{1}{r_{0}^{-}}+\left[ -c+\frac{w_{3}}{t}\underset{0}{%
\overset{t}{\int }}\frac{ds}{v+D_{3}}\right] t <\frac{1}{r_{0}^{-}}-\frac{c}{2%
}t,\ \ \ \ \text{for all }t\in ( 0,\delta ) .
\]%
If $r_{0}^{-}$ is sufficiently large, then we can find $T^{*}\in ( 0,\delta %
) $ such that%
\[
\frac{1}{r_{0}^{-}}-\frac{c}{2}T^{*}=0.
\]%

This entails 

\begin{equation}
\psi \left( T^*\right) = \frac{1}{r_{0}^{-}}-cT^* + w_{3}\underset{0}{\overset{T^*}{\int }}\frac{ds}{v+D_{3}} < \frac{1}{r_{0}^{-}}-\frac{c}{2}T^{*} = 0.
\end{equation}

Thus one has $\psi \left( T^*\right) < 0$, but $\psi \left( 0\right) > 0$, andby  application of the mean value theorem, we obtain the existence of some $T\in ( 0,\delta)$ , $T < T^*$, s.t $\psi \left( T\right) = 0$. 
This implies the solution of \eqref{eq:x1o}-\eqref{eq:x3o} blows up in finite time, at $t=T$, and by a standard comparison
argument \cite{S83}, the solution of the corresponding PDE system \eqref{eq:(1.1)}-\eqref{eq:(1.3)}, also blows up in finite time. We can thus state the following theorem

\begin{theorem}
\label{thm:tp1}
Consider the spatially explicit three species food chain model \eqref{eq:(1.1)}-\eqref{eq:(1.3)}. For $c < \frac{w_3}{D_3}$, $r$ blows up in finite time, that is

\begin{equation}
\label{eq:bu}
\mathop{\lim}_{t \rightarrow T^{**}} ||r(t)||_{\infty} \rightarrow \infty, 
 \end{equation}
  as long as the initial data $v_0, r_0$ are large enough. Here  $T^{**} < T < \infty$.
\end{theorem}

Note the above argument easily generalises to the case $c < k \left(\frac{w_3}{D_3}\right)$, where 

\begin{equation}
0< k = \left(\frac{w_0 b_1 D_3}{w_1\left( a_1 + \frac{(a_1)^2}{4a_2}\right) + w_0 b_1 D_3}\right) < 1.
\end{equation}

Thus one can also state the following corollary

\begin{corollary}
\label{cor:c2}
Consider the three species food chain model \eqref{eq:(1.1)}-\eqref{eq:(1.3)}. Even if

\begin{equation}
\label{eq:l1r}
c < \left(  \frac{w_0 b_1 D_3}{w_1\left( a_1 + \frac{(a_1)^2}{4a_2}\right) + w_0 b_1 D_3}\right) \frac{w_3}{D_3},
\end{equation}

solutions to \eqref{eq:(1.1)}-\eqref{eq:(1.3)} with certain initial data are not bounded forward in time. In fact the solution $r$ to \eqref{eq:(1.3)} can blow-up in finite time, that is

\begin{equation}
\label{eq:buc}
\mathop{\lim}_{t \rightarrow T^{***}} ||r(t)||_{\infty} \rightarrow \infty, 
 \end{equation}
 as long as the initial data $v_0, r_0$ are large enough. Here  $T^{***} < T < \infty$. 
\end{corollary}

\begin{remark}
We remark that the methods of this section can be directly applied to prove blow up in the ODE case as well. However the earlier proof via Theorem \ref{thm:r1} has the advantage, that we can explicitly give a sufficient condition on the largeness of the data, required for blow-up.
Also, not just the $L^{\infty}(\Omega)$ norm, but every $L^p(\Omega)$ norm, $p \geq 1$, blows up. This is easily seen in analogy with the equation $r_t=d_3\Delta r + \delta r^2$, and an application of the first eigenvalue method. Also note the blow-up times $T^*, T^{**}$ for the PDE case are not to be confused with the blow-up times $T^*, T^{**}$ for the ODE case.  
\end{remark}

\section{Numerical validation}

In this section we numerically simulate the ODE system \eqref{eq:x1o}-\eqref{eq:x3o}, as well as the PDE system \eqref{eq:(1.1)}-\eqref{eq:(1.3)}, (in 1d and 2d), in order to validate our results Theorem \ref{thm:r1}, Theorem \ref{thm:d1}, Theorem \ref{thm:tp1} and Corollary \ref{cor:c2}. To this end we select the following parameter range,

 \begin{eqnarray}
 \label{eq:para}
 && a_1=1,b_1=0.5,w_0=0.55,D_0=10,a_2=1,w_1=0.1,D_1=13,w_2=0.25,D_2=10, \\ \nonumber 
 &&  c=0.055,w_3=1.2,D_3=20.
 \end{eqnarray}
 
  These parameters satisfy condition \eqref{eq:ac1}, from \cite{AA02}. 
  
  \begin{equation}
\label{eq:lnv}
c = 0.055 < 0.0587 = \left(  \frac{w_0 b_1 D_3}{w_1\left( a_1 + \frac{(a_1)^2}{4a_2}\right) + w_0 b_1 D_3}\right) \frac{w_3}{D_3}.
\end{equation}

Despite this, we see finite time blow-up. The systems are simulated in MATLB R2011a. For simulation of the ODE systems we have used the standard $ode45$ routine which uses a variable time step Runge Kutta method. To explore the spatiotemporal dynamics of the PDE system in one and two dimensional spatial domain, the system of partial differential equations is numerically solved using a finite difference method. A central difference scheme is used for the one dimensional diffusion term, whereas standard five point explicit finite difference scheme is used for the two dimensional diffusion
terms. The system is studied with positive initial condition and Neumann boundary condition in the spatial domain $0 \leq x \leq L_{x}$, $0 \leq y \leq L_y$, where $L_x=L_y=\pi$. 
Note, our proof of blow-up, allows for Dirichlet, Neumann or Robin type boundary conditions.
Simulations are done over this square domain with spatial resolution $\Delta x = \Delta y = 0.01$, and time step size $\Delta t = 0.01$. We next present the results of our simulations.

\begin{figure}[!ht]
	\begin{center}
	 {\includegraphics[scale=0.30]{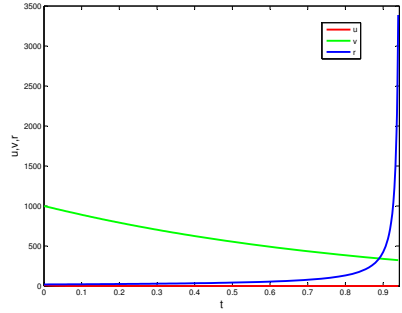}}
		\end{center}
\caption{These figures show blow up in the ODE case.}
\label{fig:1}
\end{figure}

\begin{figure}[!ht]
	\begin{center}
	{\includegraphics[scale=0.30]{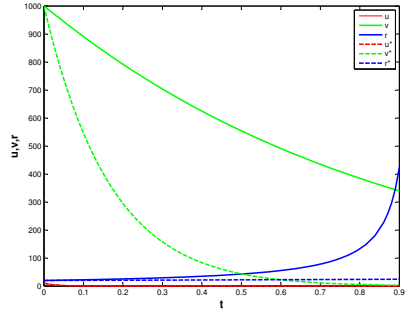}}
		\end{center}
\caption{We also compare the original system with the modified system to illustrate our idea, to show blowup.}
\label{fig:ODE}
\end{figure}

\begin{figure}[!ht]
	\begin{center}
		{\includegraphics[scale=0.30]{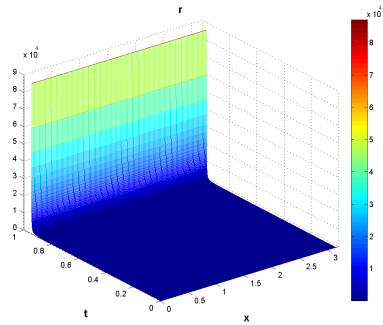}}
		\end{center}
\caption{These figures show blow up in the 1d PDE case. }
\label{fig:PDE1}
\end{figure}

%
%

\begin{figure}[!ht]
	\begin{center}
		{\includegraphics[scale=0.3]{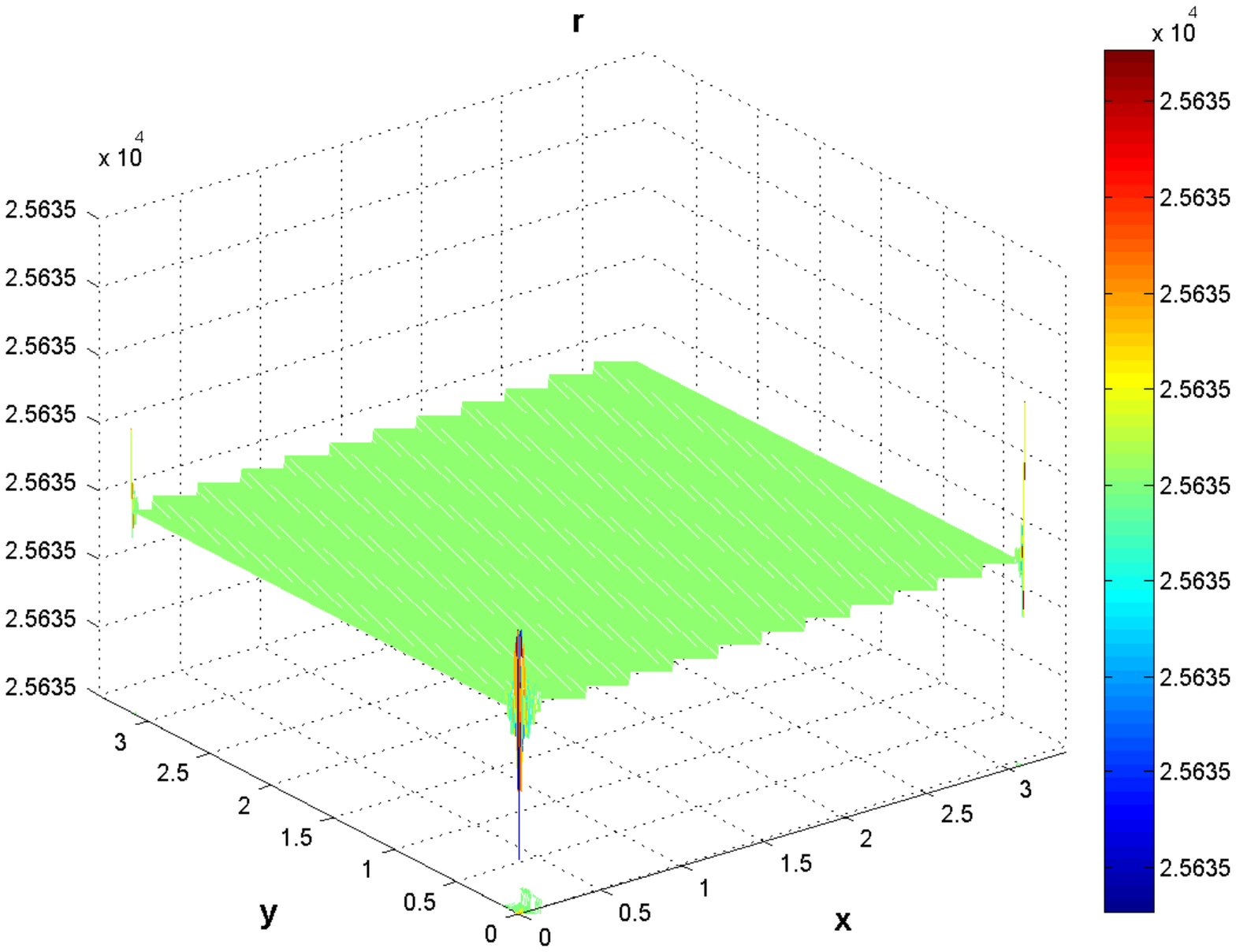}}
			\end{center}
		\caption{These figures show a surface plot of blowup in the 2d PDE case .}
		\label{fig:PDE2}
	\end{figure}

\section{Conclusion}
In the current letter we have shown that the solutions to the system \eqref{eq:x1o}-\eqref{eq:x3o}, modeling a tri-trophic food chain can exhibit finite time blow-up under the condition \eqref{eq:ac1} from theorem \ref{thm:aziz}, as long as the initial data is large enough. This is also true in the case of the spatially explicit model.  
 Thus the basin of attraction of the invariant set $\mathcal{A}$, explicitly constructed in \cite{AA02}, \emph{is not all of} $\mathbb{R}^3_{+}$, as claimed in \cite{AA02}. Furthermore system \eqref{eq:x1o}-\eqref{eq:x3o} \emph{is not dissipative} in all of $\mathbb{R}^3_{+}$, also as claimed in \cite{AA02}. For a numerical valiadation of these results please see figures \ref{fig:ODE}, \ref{fig:PDE1}, \ref{fig:PDE2}.

However, the model posesses very rich dynamics, in the parameter region

\begin{equation}
\label{eq:gc}
\frac{w_3}{v+D_3} < c < \frac{w_3}{D_3}.
\end{equation}

Thus an extremely interesting open question is, what is the basin of attraction 
for an appropriately defined and constructed $\mathcal{A}$?
This is tantamount to asking, which sorts of initial data lead to globally existing solutions, under the dynamics of \eqref{eq:x1o}-\eqref{eq:x3o}, and the parameter range \eqref{eq:gc}? The same questions can be asked, in the case of the spatially explicit model.

\section{Acknowledgment} 
The present research of NK is supported by UGC under Raman fellowship, Project no. 5-63/2013(c) and IIT Mandi under the project no. IITM/SG/NTK/008 and DST under IU-ATC phase 2, Project no. SR/RCUK-DST/IUATC Phase 2/2012-IITM(G).

 \end{document}